\documentclass[12pt,a4paper]{amsart}

\usepackage[a4paper,top=2cm,bottom=2cm,left=2cm,right=2cm]{geometry}
\usepackage[english]{babel}
\usepackage{newlfont}
\usepackage{fancyhdr}
\usepackage{graphicx}
\usepackage{amsmath}
\usepackage{amsfonts}
\usepackage{amsthm}
\usepackage{latexsym}
\usepackage[utf8]{inputenc}
\usepackage{enumerate}
\usepackage{stmaryrd}
\usepackage{tikz}
\usetikzlibrary{matrix,arrows}

\theoremstyle{plain}

\newtheorem{thm}{Theorem}[section]

\newtheorem{lem}[thm]{Lemma}

\theoremstyle{definition}

\title{On the Splitting Principle of Beniamino Segre}

\author{Camilla Felisetti and Claudio Fontanari}
\date{}
\thanks{This research project was partially supported by GNSAGA of INdAM and by PRIN 2017 ``Moduli Theory and Birational Classification".  
}

\begin{document}

\begin{abstract}
We state and prove in modern terms a Splitting Principle first claimed 
by Benia\-mino Segre in 1938, which should be regarded as a strong form of the 
classical Principle of Connectedness.
\end{abstract}

\maketitle

\section{Introduction}

The memoir \cite{Segre38}, appeared in the \emph{Annali di Matematica Pura e Applicata} 
in 1938, is the last paper published in Italy by Beniamino Segre before the end of 
Second World War, as a consequence of the odious racial laws promulgated by the 
Italian Fascist Regime in that same year. 

The title \emph{Un teorema fondamentale della geometria sulle superficie algebriche 
ed il principio di spezzamento} alludes to a Splitting Principle stated there for the first 
time (albeit with a sloppy formulation and an incomplete proof). Francesco Severi, who 
praised that work as the best one of his former student (see \cite{BC98}, p. 233), 
extensively reconsidered it in his treatise \cite{Severi59}, published in 1959. 

Such Splitting Principle should be regarded as a strong form of the classical Principle 
of Connectedness, attributed by Segre to Enriques and by Severi to Noether, stating 
that if the general member of a flat family $\{ X_t \}$ of closed subschemes of 
$\mathbb{P}^k$ parameterized by an irreducible curve $T$ of finite type is connected, 
then $X_t$ is connected for all $t \in T$ (see for instance \cite{Hartshorne77}, 
III, Exercise 11.4 on p. 281, or \cite{Sernesi09}, Proposition 6.5). 

In modern terms, we state it as follows (see \cite{Segre38}, p. 111, and \cite{Severi59}, pp. 81-82): 

\begin{thm}[Splitting Principle]\label{splitting}
	Let $\{ E \}$ be a flat family over a normal base of nodal curves on a smooth surface $F$ of geometric genus $p_g$. Suppose that the general element $E$ of $\{ E \}$ is irreducible and that a special element $E_0$ of $\{ E \}$ splits as $E_0=C+D$ with $C,D$ irreducible. 
	Let $\Gamma := C\cap D = \Gamma_1 \sqcup \Gamma_2$, where $\Gamma_1$ is the set of points which are limits of nodes of the general curve in $\{E \}$ and $\Gamma_2$ is its complement in $\Gamma$. 
	Assume that $\vert D_{\mid D}\vert$ is non-empty and $C$ is sufficiently general with respect to $D$, in particular that $\vert C(-\Gamma_1)\vert$ has no base points on $D$.
	If $c_i$ is the cardinality of $\Gamma_i$
	then we have $c_2\geq p_g+1$, unless the points in $\Gamma_2$ are linearly dependent with respect to $K_F$.
	
\end{thm}		

 The assumptions that all curves in $\{ E \}$ are nodal and that
 $C$ is general with respect to $D$ are both missing from Segre's 
 statement in \cite{Segre38} and are added by Severi in \cite{Severi59}, p. 81. 
 We point out that the splitting of $\Gamma$ into the disjoint union of $\Gamma_1$ 
 and $\Gamma_2$ is well-defined only if $E_0$ is assumed to have double points 
 as singularities (hence it is implicit in Segre's argument, see in particular \cite{Segre38}, 
 \S 10, p. 122: "\emph{I punti di $\Gamma$ si potranno allora distinguere in due categorie, 
 secondoch\`e provengono o meno come limiti da punti doppi di $E$, ossia rispettivamente a seconda 
 che non risultano oppure risultano punti di collegamento tra $C$ e $D$; denotiamo ordinatamente 
 con $\Gamma_1$, $\Gamma_2$ i gruppi costituiti dai punti del primo o del secondo tipo,} 
 (...) \emph{talch\`e sar\`a $\Gamma=\Gamma_1 + \Gamma_2$}").
 Furthermore, Severi's statement in \cite{Severi59}, p. 81, assumes that the curve $D$ is 
 \emph{ordinaria}, which in particular implies that the characteristic series $E_{\mid D}$ 
 is complete on $D$. Severi in \cite{Severi59}, p. 197, comments: "(...) \emph{abbiamo stabilito 
 nel n. 23 il notevole principio di spezzamento di B. Segre contenuto nella Memoria degli Annali 
 di Matematica, 1938, per\`o sotto l'ipotesi, qui aggiunta, che 
 una componente del limite della curva che tende a spezzarsi sia 
 una curva ordinaria irriducibile (appartenente cio\`e totalmente 
 ad un sistema irriducibile avente su di essa la serie caratteristica 
 completa). Siccome questo principio fa entrare in giuoco soltanto il genere 
 geometrico $p_g$ della superficie $F$ e non l'irregolarit\`a, \`e ragionevole supporre
 (n. 23) che il suo fondamento topologico sia in relazione soltanto ai cicli 
 bidimensionali della riemanniana di $F$; epper\`o, siccome esso \`e vero, qualunque sia
$p_g$, sopra le superficie regolari, e per $p_g=0$ su ogni superficie irregolare
 apparisce naturale di supporre che il principio stesso sia vero sempre.}"
 Finally, in \cite{Severi59}, p. 81, the curve $D$ is assumed to be nonsingular 
 ("\emph{priva di nodi}"), but our modern proof shows that also this assumption 
 is unnecessary. Indeed, the main ingredients are Riemann-Roch theorem, Serre duality 
 and adjunction formula, which hold also in the singular case up to replacing the canonical 
 bundle $K_D$ with the dualizing sheaf $\omega_D$ (see for instance \cite{BHPV04}, p. 62).
 We are going to apply these formulas to restrictions to $D$ of divisors 
 on the smooth surface $F$, hence to Cartier divisors on the nodal curve $D$. 

\section{The proofs}

Our proof of Theorem \ref{splitting} relies on a couple of crucial remarks. 

\begin{lem}\label{nonempty}
We have $\Gamma_2 \neq \emptyset$.
\end{lem}

\begin{proof}
If $\Gamma_2 = \emptyset$ then all nodes of $E_0 = C+D$ belong to $\Gamma_1$, i.e. they are
limits of nodes of $E$. Hence by \cite{T76}, Th\'eor\`eme 1 on p. 73, locally around $E_0$ we may 
resolve simultaneously all singularities of $E$. In this way we would obtain a family of irreducible 
curves degenerating to a disconnected one, contradicting the classical Principle of 
Connectedness (see for instance \cite{Hartshorne77}, III, Exercise 11.4 on p. 281).
\end{proof}

\begin{lem}\label{severi}
There is at least one point $P \in \Gamma_2$ which 
	   is not a base point of the complete linear series
	   $\vert {E_0}_{\mid D} -\Gamma_1\vert$ on $D$.
\end{lem}

\begin{proof}
       Assume by contradiction that 
	   \begin{equation}\label{generic}
	   h^0(D, {E_0}_{\mid D} - \Gamma_1) = h^0(D, {E_0}_{\mid D}-\Gamma_1 - \Gamma_2) 
	   = h^0(D, {E_0}_{\mid D}-\Gamma).
	   \end{equation}
	   On the other hand, since $E_0 = C+D$ we have
	   $\vert {E_0}_{\mid D}- \Gamma_1\vert \supseteq \vert C_{\mid D} - \Gamma_1 \vert
	   + \vert D_{\mid D}\vert$.
	   Moreover, since $C_{\mid D} - \Gamma_1 = \Gamma_2 \neq \emptyset$ by Lemma \ref{nonempty}
	   and $\vert C(-\Gamma_1) \vert$ has no base points on $D$ by assumption, 
	   we have $h^0(D, C_{\mid D} - \Gamma_1) \geq 2$. 
	   Hence we deduce 
	   \begin{eqnarray*}
	   h^0(D, {E_0}_{\mid D}-\Gamma_1) &\ge& h^0(D, C_{\mid D} - \Gamma_1)
	   + h^0(D, D_{\mid D}) - 1 
	   \geq h^0(D, D_{\mid D}) + 1 \\
	   &=& h^0(D, {E_0}_{\mid D} - C_{\mid D}) + 1 
	   = h^0(D, {E_0}_{\mid D} - \Gamma) + 1,
	   \end{eqnarray*}
	   contradicting (\ref{generic}), so the claim is established.
\end{proof}

	\begin{proof}[Proof of Theorem \ref{splitting}]
	   We follow Segre's approach in \cite{Segre38}.
	   Let $d:= h^1(D,D_{\mid D})$. We have two possibilities: 
	   (i) $c_2\geq d+1$ or (ii) $c_2\leq d$.
	   
		\begin{enumerate}[(i)]
			\item Suppose $c_2\geq d+1$. Let $i:=h^2(F,D)$. We first prove that 
			\begin{equation} \label{specialitysing}
			d\geq p_g-i.
			\end{equation}
			Indeed, by adjunction $K_{F_{\mid D}}=\omega_D-D_{\mid D}$ and 
			by Serre duality
			$$d=h^1(D,D_{\mid D})=h^0(D,\omega_D-D_{\mid D})=h^0(D,K_{F_{\mid D}}).$$
			The short exact sequence on $F$
			$$0\to K_F(-D)\to K_F\to K_{F\mid D}\to 0$$
			yields a long exact sequence
			$$ 0\to H^0(K_F(-D))\to H^0(K_F)\to H^0(K_{F_{\mid D}})\to \ldots$$
			hence $p_g \leq i + d$.
			
			If $i=0$, we immediately get $c_2\geq d+1\geq p_g+1$.
			If instead $i>0$, then the points in $\Gamma_2$ are dependent with respect to $K_F$, i.e. $h^0(F,K_F(-\Gamma_2))>p_g-c_2$.
			Indeed, on the one hand 
			by \eqref{specialitysing} we have
			\begin{equation}\label{eq2sing}
			 p_g-c_2\leq p_g-d-1 \leq p_g-p_g+i-1=i-1.
			\end{equation}
			On the other hand, since any global section of $K_F$ which vanishes on $D$ vanishes in particular on $\Gamma_2$, we have
			$$h^0(F,K_F(-\Gamma_2))\geq h^0(F,K_F(-D))= h^2(F, D) = i.$$
			By $\eqref{eq2sing}$ we conclude that $h^0(F,K_F(-\Gamma_2))\geq i >p_g-c_2$.
		\item Suppose $c_2\leq d$.
		    Let $P\in \Gamma_2$ and set $\Gamma_2^*:=\Gamma_2 \setminus P$.\footnote{Note that $\Gamma_2$ is non empty by Lemma \ref{nonempty}, but $\Gamma_2^*$ might be.}
	    	Observe first that the linear series 
	    	$$\vert D_{\mid D}+\Gamma_2^* \vert$$
		on $D$ is special. 
		In fact, 
		$h^1(D,D_{\mid D}+\Gamma_2^*)= h^0(D, \omega_D-D_{\mid D}-\Gamma_2^*)\geq  h^0(D,\omega_D-D_{\mid D})-c_2+1=h^1(D,D_{\mid D})-c_2+1=d-c_2+1\geq 1.$
		In particular, by adjunction we have that $H^0(D,\omega_D-D_{\mid D}-\Gamma_2^*)\cong H^0(D,K_{F_{\mid D}}-\Gamma_2^*)$ is non-zero. 
		
		We are going to prove that the natural inclusion 
	    	$$H^0(F,K_F-\Gamma_2) \subseteq H^0(F,K_F-\Gamma_2^*)$$
	    	is an isomorphism for some choice of $P\in \Gamma_2$, i.e. that the points in $\Gamma_2$ are dependent with respect to $K_F$. 
	    	
		Indeed, by Lemma \ref{severi}, there exists at least one $P \in \Gamma_2$ 
		such that the complete linear series 
		$$\vert E_{0\mid D}-\Gamma_1\vert
		= \vert {C}_{\mid D}+{D}_{\mid D} -\Gamma_1\vert =
	   \vert {D}_{\mid D} + \Gamma -\Gamma_1\vert =
	   \vert D_{\mid D}+ \Gamma_2 \vert$$
		on $D$ does not admit $P$ as a base point. 
		 
		 On the other hand, by the Riemann-Roch theorem
		 $$h^0(D_{\mid D}+\Gamma_2)=h^1(D_{\mid D}+\Gamma_2)+\deg (D_{\mid D}+\Gamma_2) + 1-p_a(D)$$
		 $$h^0(D_{\mid D}+\Gamma_2^*)=h^1(D_{\mid D}+\Gamma_2^*)+\deg (D_{\mid D}+\Gamma_2)-1+1-p_a(D).$$
		 Since $h^0(D_{\mid D}+\Gamma_2^*)=h^0(D_{\mid D}+\Gamma_2)-1$ then $h^1(D_{\mid D}+\Gamma_2)=h^1(D_{\mid D}+\Gamma_2^*)$ and by Serre duality
		\begin{equation}\label{eq3sing} h^0(D,\omega_D-D_{\mid D}-\Gamma_2^*) = h^0(D,\omega_D-D_{\mid D}-\Gamma_2).
		\end{equation}
		
		Suppose now by contradiction that the inclusion $H^0(F,K_F-\Gamma_2) \subseteq H^0(F,K_F-\Gamma_2^*)$
		is strict, i.e. that there exists an effective divisor $A$ in $\mathbb{P}H^0(F,K_F-\Gamma_2^*)$ not passing through $P$.
		Note that $A \cap D \neq D$, since $P\in D\setminus A$.
		
		Now, if $\Gamma_2^*$ is not empty, then 
		$A \cap D \neq \emptyset$, 
		 since $\emptyset\neq \Gamma_2^*\subset A\cap D$, and $A_{\mid D}$ is a nontrivial effective divisor on $D$ lying in $\mathbb{P}H^0(D,K_{F_{\mid D}}-\Gamma_2^*)\setminus \mathbb{P}H^0(D,K_{F_{\mid D}}-\Gamma_2)$, contradicting \eqref{eq3sing}.
		
		The same conclusion holds if $\Gamma_2^*$ is empty but $A \cap D \neq \emptyset$. On the other hand, if $\Gamma_2^*$ is empty and $A_{\mid D} = 0$, then 
		$K_{F_{\mid D}} \cong \mathcal{O}_D$ and by adjunction we have $\omega_D = D_{\mid D}$. Hence \eqref{eq3sing} 
		implies 
		$$
		1 = h^0(D,\mathcal{O}_D) = h^0(D,\mathcal{O}_D(-P)) = 0
		$$
		and this contradiction ends the proof. 
			\end{enumerate}
		\end{proof}

\vspace{0.5cm}

\noindent
Camilla Felisetti \newline
Dipartimento di Matematica \newline
Universit\`a di Trento \newline
Via Sommarive 14 \newline
38123 Trento, Italy. \newline
E-mail: camilla.felisetti@unitn.it

\vspace{0.5cm}

\noindent
Claudio Fontanari \newline
Dipartimento di Matematica \newline
Universit\`a di Trento \newline
Via Sommarive 14 \newline
38123 Trento, Italy. \newline
E-mail: claudio.fontanari@unitn.it

\end{document}